\documentclass[11pt]{article}
\usepackage{amsmath}
\usepackage{amssymb}
\usepackage{amsthm}
\newtheorem{theorem}{Theorem}[section]
\newtheorem{lemma}[theorem]{Lemma}
\newtheorem{corollary}[theorem]{Corollary}
\newtheorem{proposition}[theorem]{Proposition}

\newtheorem*{thmA}{Theorem A}
\newtheorem*{thmB}{Theorem B}

\numberwithin{equation}{section}

\long\def\symbolfootnote[#1]#2{\begingroup%
\def\thefootnote{\fnsymbol{footnote}}\footnote[#1]{#2}\endgroup}

\begin{document}

\def\C{{\mathbb C}}
\def\N{{\mathbb N}}
\def\Z{{\mathbb Z}}
\def\R{{\mathbb R}}
\def\K{{\mathbb K}}
\def\T{{\mathbb T}}
\def\F{{\cal F}}
\def\U{{\cal U}}
\def\M{{\cal M}}
\def\Q{{\cal Q}}
\def\H{{\cal H}}
\def\CC{{\cal C}}
\def\E{{\cal E}}
\def\rr{{\cal R}}
\def\pp{{\cal P}}
\def\square{\vrule height6pt width6pt depth 0pt}
\def\epsilon{\varepsilon}
\def\kappa{\varkappa}
\def\phi{\varphi}
\def\leq{\leqslant}
\def\geq{\geqslant}
\def\re{\text{\tt Re}\,}
\def\slim{\mathop{\hbox{$\overline{\hbox{\rm lim}}$}}\limits}
\def\ilim{\mathop{\hbox{$\underline{\hbox{\rm lim}}$}}\limits}
\def\supp{\hbox{\tt supp}\,}
\def\dim{\hbox{\tt dim}\,}
\def\ker{\hbox{\tt ker}\,}
\def\Ker{\hbox{\tt ker}^\star\,}
\def\Int{\hbox{\tt int}\,}
\def\spann{\hbox{\tt span}\,}
\def\Re{\hbox{\tt Re}\,}
\def\det{\hbox{\tt det}\,}
\def\deg{\hbox{\tt deg}\,}
\def\ssub#1#2{#1_{{}_{{\scriptstyle #2}}}}
\def\summ{\sum\limits}
\def\maxx{\max\limits}
\def\minn{\min\limits}
\def\limm{\lim\limits}
\def\ud{\overline{d}}
\def\ld{\underline{d}}
\font\Goth=eufm10 scaled 1440
\def\uu{\hbox{{\Goth u}}}
\def\bbeta{\overline{\beta}}

\title{On the spectrum of frequently hypercyclic operators}

\author{Stanislav Shkarin}

\date{}

\maketitle

\begin{abstract} A bounded linear operator $T$ on a Banach
space $X$ is called frequently hypercyclic if there exists $x\in X$
such that the lower density of the set $\{n\in\N:T^nx\in U\}$ is
positive for any non-empty open subset $U$ of $X$. Bayart and
Grivaux have raised a question whether there is a frequently
hypercyclic operator on any separable infinite dimensional Banach
space. We prove that the spectrum of a frequently hypercyclic
operator has no isolated points. It follows that there are no
frequently hypercyclic operators on all complex and on some real
hereditarily indecomposable Banach spaces, which provides a negative
answer to the above question.
\end{abstract}


\normalsize

\section{Introduction \label{s1}}\rm

All vector spaces in this article are assumed to be over the field
$\K$ being either the field $\C$ of complex numbers or the field
$\R$ of real numbers. As\symbolfootnote[0]{Partially supported by
Plan Nacional I+D+I grant no. MTM2006-09060 and Junta de
Andaluc\'{\i}a FQM-260.} usual, $\T=\{z\in\C:|z|=1\}$, $\Z$ is the
set of integers, $\Z_+$ is the set of non-negative integers and $\N$
is the set of positive integers. For a Banach space $X$, $L(X)$
stands for the algebra of bounded linear operators on $X$ and $X^*$
is the space of continuous linear functionals on $X$.

Recall that a continuous linear operator $T$ on a topological vector
space $X$ is called {\it hypercyclic} if there is $x\in X$ such that
the orbit $\{T^nx:x\in\Z_+\}$ is dense in $X$. We refer to surveys
\cite{ge1,ge2} for additional information on hypercyclicity. There
are several stronger versions of the above property of linear
operators like hereditarily hypercyclic operators, topologically
mixing operators and operators satisfying the Kitai Criterion.
Recently Bayart and Grivaux \cite{bagri} have introduced the concept
of frequently hypercyclic operators. Recall that the {\it upper} and
{\it lower} densities of a set $A\subset\Z_+$ are defined by the
formula
$$
\ld(A)=\ilim_{n\to\infty}\frac{|\{m\in A:m\leq n\}|}{n},\qquad
\ud(A)=\slim_{n\to\infty}\frac{|\{m\in A:m\leq n\}|}{n},
$$
where $|B|$ stands for the number of elements of a finite set $B$.

\medskip

\noindent{\bf Definition 1.} \ Let $T$ be a continuous linear
operator on a topological vector space $X$. Then $x\in X$ is called
a {\it frequently hypercyclic} vector for $T$ if the lower density
of the set $\{n\in\N:T^nx\in U\}$ is positive for any non-empty open
subset $U$ of $X$. An operator $T$ is called {\it frequently
hypercyclic} if it has a frequently hypercyclic vector.

We also say that $x\in X$ is an {\it $\uu$-frequently hypercyclic}
vector for $T$ if the upper density of the set $\{n\in\N:T^nx\in
U\}$ is positive for any non-empty open subset $U$ of $X$. An
operator $T$ is called {\it $\uu$-frequently hypercyclic} if it has
an $\uu$-frequently hypercyclic vector. We denote the set of
$\uu$-frequently hypercyclic vectors for $T$ by the symbol $\uu(T)$.

\medskip

Clearly frequent hypercyclicity implies $\uu$-frequent
hypercyclicity, which, in turn, implies hypercyclicity. Bayart and
Grivaux \cite{bagri} have raised the following question.

\medskip

\noindent{\bf Question 1.} \ Does every separable infinite
dimensional Banach space support a frequently hypercyclic operator?

\medskip

This question is partially motivated by the following chain of
results. Ansari \cite{ansa1} and Bernal-Gonz\'ales \cite{bernal}
showed independently that for any separable infinite dimensional
Banach space $X$ there is a hypercyclic operator $T\in L(X)$.
Grivaux \cite{gri} proved that there is a mixing operator on any
separable infinite dimensional Banach space. Finally, in \cite{shka}
it is shown that on any separable infinite dimensional Banach space,
there is a bounded linear operator, satisfying the Kitai Criterion.
In this article we show that unlike for other strong forms of
hypercyclicity, the answer to the above question is negative.

\begin{theorem}\label{main} Let $T$ be a bounded linear operator on
a separable infinite dimensional Banach space $X$ such that
$\limm_{n\to \infty}\|T^n\|^{1/n}=0$. Then $I+T$ is not
$\uu$-frequently hypercyclic. In particular, $I+T$ is not frequently
hypercyclic.
\end{theorem}

In the case $\K=\C$, we use the above theorem to prove the following
result.

\begin{theorem} \label{spec} Let $X$ be a complex Banach space and
$T\in L(X)$ be frequently hypercyclic. Then the spectrum $\sigma(T)$
of $T$ has no isolated points.
\end{theorem}

This result shows that the class of frequently hypercyclic operators
is significantly smaller than the class of hypercyclic operators.
For instance, in \cite{grish} it is demonstrated that if $X$ is a
separable Banach space and $M$ is the operator norm closure of the
set of finite rank nilpotent operators, then the set $\{T\in M:I+T\
\ \text{is hypercyclic}\}$ is a dense $G_\delta$ subset of $M$.
Thus, in a sense, a 'generic' operator with one-point spectrum
$\{1\}$  is hypercyclic.

We use Theorem~\ref{main} to obtain a negative answer to Question~1.
It turns out that there are no frequently hypercyclic operators on
Banach spaces with few operators. Recall that a bounded linear
operator $T$ on a Banach space $X$ is called {\it strictly singular}
if the restriction of $T$ to any closed infinite dimensional
subspace of $X$ is not an isomorphism onto its image. As is
well-known \cite{pich}, the set ${\cal S}(X)$ of strictly singular
operators is an ideal in the algebra $L(X)$.

By $\Phi_\K$ we denote the class of infinite dimensional Banach
spaces $X$ over the field $\K$ for which $L(X)/{\cal S}(X)$ is
isomorphic to $\K$. It is straightforward to see that $X\in \Phi_\K$
if and only if any $T\in L(X)$ has the shape $T=zI+S$ for $z\in\K$
and $S\in {\cal S}(X)$.

\begin{theorem} \label{hihi} Let $X\in\Phi_\K$. Then
there is no bounded $\uu$-frequently hypercyclic operator on $X$. In
particular, $X$ does not support a frequently hypercyclic operator.
\end{theorem}

The above theorem gives a negative answer to Question~1 provided
there are separable Banach spaces in classes $\Phi_\C$ and
$\Phi_\R$. Recall that an infinite dimensional Banach $X$ space is
called {\it hereditarily indecomposable} or an {\it HI space} if for
any infinite dimensional closed linear subspaces $L$ and $M$ of $X$
satisfying $L\cap M=\{0\}$ their sum $L+M$ is not closed. Separable
HI Banach spaces were first constructed by Gowers and Maurey
\cite{hi1}, see also the survey \cite{hi2}. One of the most
interesting features of HI spaces $X$ is \cite{hi1,fer1} that any
complex HI space belongs to $\Phi_\C$. It is also worth noting that
there are separable Banach spaces in $\Phi_\C$ that are not HI
\cite{aaa}. Anyway, we have the following corollary.

\begin{corollary}\label{hihi1} Let $X$ be a separable complex HI
space. Then there is no bounded $\uu$-frequently hypercyclic
operator on $X$. In particular, $X$ does not support a frequently
hypercyclic operator.
\end{corollary}

The case of real HI spaces is slightly more complicated. Namely,
Ferenczi \cite{fer1,fer2} demonstrated that if $X$ is a real HI
space, then $L(X)/{\cal S}(X)$ is isomorphic to either $\R$, or $\C$
or the quaternion division ring $\mathbb H$. Moreover, in
\cite{fer2} it is shown that all three possibilities do occur on the
level of separable real HI spaces and is observed that the original
real separable HI space $X$ of Gowers and Maurey \cite{hi1}
satisfies $L(X)/{\cal S}(X)\simeq \R$ and therefore belongs to
$\Phi_\R$. This remark together with Theorem~\ref{hihi} gives us the
following corollary.

\begin{corollary}\label{hihi2} There is an infinite dimensional separable
real Banach space $X$, which does not support a $\uu$-frequently
hypercyclic operator. In particular, $X$ does not support a
frequently hypercyclic operator.
\end{corollary}

\section{Proof of Theorem~\ref{main}}

\begin{theorem} \label{main1} Let $T$ be a continuous linear
operator on a real topological vector space $X$ and $x\in X$ be such
that there exists $f\in X^*\setminus\{0\}$ satisfying
\begin{equation}\label{lim}
\lim_{n\to\infty} |f((T-I)^nx)|^{1/n} =0 \ \ \ \text{for each}\ \ \
x\in X.
\end{equation}
Then $x$ is not a $\uu$-frequently hypercyclic vector for $T$.
\end{theorem}

\begin{proof} Assume that $x\in\uu(T)$. First, observe, that we can find $u\in
X$ such that $f(u)<0$ and $f(Tu)>0$. Indeed, otherwise, using
continuity of $f$ and $T$, we see that $f(Tu)\leq 0$ if $f(u)\leq
0$. Since $f(-y)=-f(y)$ for any $y\in X$, we also have $f(Tu)\geq 0$
if $f(u)\geq 0$. Hence $T(H)\subseteq H$, where $H=\ker f$. Thus,
$T$ has a closed invariant subspace of codimension 1. On the other
hand, in \cite{joh} it is proven that any hypercyclic operator on
any real or complex topological vector space has no closed invariant
subspaces of finite codimension. This contradiction shows that there
is $u\in X$ such that $f(u)<0$ and $f(Tu)>0$. Since $f$ and $T$ are
continuous, we can choose an open neighborhood $U$ of $u$ such that
\begin{equation} \label{pm}
f(v)<0\ \ \ \text{and}\ \ \ f(Tv)>0 \ \ \ \text{for any $v\in U$.}
\end{equation}
Since $x\in\uu(T)$, there exists a subset $A$ of $\Z_+$ of positive
upper density such that
\begin{equation}\label{pm0}
\text{$f(T^nx)\in U$ for any $n\in A$.}
\end{equation}
Consider now the series
\begin{equation} \label{ps}
F(z)=f(x)+\sum_{n=1}^\infty
\frac{f((T-I)^nx)}{n!}z(z-1)\dots(z-n+1), \quad z\in\C.
\end{equation}
Condition (\ref{lim}) implies that the series in (\ref{ps}) is
uniformly convergent on any compact subset of $\C$ and therefore
defines an entire function $F$. Moreover, $F$ has exponential type
0, that is,
\begin{equation}\label{type}
\lim_{R\to\infty}\frac{\ln M_F(R)}{R}=0,\ \ \text{where}\ \
M_F(R)=\max_{|z|\leq R}|F(z)|.
\end{equation}
Indeed, by (\ref{lim}), for any $\epsilon\in(0,1)$, there is
$m\in\N$ such that $|f((T-I)^nx)|<\epsilon^n$ for $n>m$. Then
\begin{align*}
M_F(R)&\leq |f(x)|+\sum_{n=1}^\infty
\frac{|f((T-I)^nx)|}{n!}R(R+1)\dots(R+n-1)\leq
\\
&\leq1+\sum_{n=1}^\infty
\frac{\epsilon^n}{n!}R(R+1)\dots(R+n-1)+O(R^m)=(1-\epsilon)^{-R}+O(R^m)=
\\
&=e^{a(\epsilon)R}+O(R^m)=O(e^{a(\epsilon)R})\ \ \text{as
$R\to\infty$, where $a(\epsilon)=-\ln(1-\epsilon)$}.
\end{align*}
Since $a(\epsilon)\to0$ as $\epsilon\to 0$, (\ref{type}) follows.
From the definition of $F$ it also easily follows that
\begin{equation}\label{fF}
\text{$F(k)=f(T^kx)$ for any $k\in\Z_+$}.
\end{equation}
Using (\ref{fF}), (\ref{pm0}) and (\ref{pm}), we see that $F(n)<0$
and $F(n+1)>0$ for any $n\in A$. Since $F$ is real on the real axis,
for any $n\in A$, there is $t_n\in (n,n+1)$ such that $F(t_n)=0$.
Let now $N_F(R)$ be the number of zeros of $F$ in the set
$\{z\in\C:0<|z|<R\}$. The last observation yields $N_F(n+1)\geq k_n$
for any $n\in\Z_+$, where $k_n=|\{m\in A:m\leq n\}|$ is the counting
function of the set $A$. On the other hand, by a corollary of
Jensen's Theorem, see \cite[p.~15]{levin}, we have
\begin{equation}\label{jen}
N_F(R)\leq\ln M_F(eR)+O(\ln R) \quad \text{as $R\to\infty$}.
\end{equation}
Hence
$$
\slim_{n\to\infty}\frac{\ln M_F(e(n+1))}{n}\geq
\slim_{n\to\infty}\frac{N_F(n+1)}{n}\geq
\slim_{n\to\infty}\frac{k_n}{n}=\ud(A)>0,
$$
which contradicts (\ref{type}). \end{proof}

\begin{corollary} \label{tstar} Let $X$ be a Banach space and $T\in L(X)$ be
such that there exists $f\in X^*\setminus\{0\}$ for which
$\limm_{n\to\infty}\|T^{*n}f\|^{1/n}=0$. Then $I+T$ is not
$\uu$-frequently hypercyclic.
\end{corollary}

\begin{proof} If $\K=\R$, we just apply Theorem~\ref{main1} to the operator $I+T$
to see that the latter has no $\uu$-frequently hypercyclic vectors.
If $\K=\C$, we consider $X$ as a real Banach space, $T$ as an
$\R$-linear operator and apply Theorem~\ref{main1} to the operator
$I+T$ and the functional ${\tt Re}\,f$.
\end{proof}

Theorem~\ref{main} follows immediately from Corollary~\ref{tstar}
since $\|T^{*n}f\|\leq \|f\|\|T^{*n}\|=\|f\|\|T^n\|$ for any $f\in
X^*$ and any $n\in\Z_+$.

\section{Rotations and powers of $\uu$-frequently hypercyclic operators \label{secrot}}

In order to reduce Theorem~\ref{hihi} and Theorem~\ref{spec} to
Theorem~\ref{main}, we need to show that the statement of
Theorem~\ref{main} remains true if we replace $I+T$ be $zI+T$ for
any $z\in\K$ with $|z|=1$. In order to do so, we need to demonstrate
that the class of $\uu$-frequently hypercyclic operators is stable
under multiplication by scalars $z$ with $|z|=1$.

For hypercyclic operators acting on complex Banach spaces, this kind
of stability was proved by Le\'on-Saavedra and M\"uller
\cite{muller}. In the case $\K=\R$ the fact that $T$ is hypercyclic
if and only if $-T$ is hypercyclic follows from the result of Ansari
\cite{ansa2}, who showed that hypercyclicity of $T$ implies
hypercyclicity of its powers. In particular $T^2$ is hypercyclic if
$T$ is. Since $(-T)^2=T^2$, the latter observation shows that $T$ is
hypercyclic if and only if $-T$ is. We are going to prove the same
stability results for $\uu$-frequently hypercyclic operators.
Concerning the powers, we shall use a result of Bourdon and Feldman
\cite{bufe}, who demonstrated that an orbit of a continuous linear
operator on a locally convex topological vector space is either
everywhere dense or nowhere dense. The only obstacle that did not
allow the same construction to work for arbitrary topological vector
spaces was that the proof of the density of the range of $p(T)$ for
any non-zero polynomial $p$ and a hypercyclic operator $T$ used the
duality argument. This obstacle was removed by Wengenroth
\cite{joh}, which allowed him to extend the Bourdon and Feldman
theorem to arbitrary topological vector spaces, see also the survey
\cite{ge2}. Throughout this section $X$ is a topological vector
space and $T\in L(X)$.

\begin{thmA} Let $x\in X$ be such that
$F^T_x=\overline{\{T^nx:n\in\Z_+\}}$ has nonempty interior. Then
$F^T_x=X$.
\end{thmA}

Although the above mentioned result of Le\'on-Saavedra and M\"uller
is formulated and proved for operators acting on Banach spaces, the
same proof with minor modifications (just replace convergent
sequences by convergent nets and use the result of Wengenroth to
ensure density of the ranges of polynomials of $T$ instead of the
duality argument) gives the following extension of their result.

\begin{thmB} Let $z\in\K$, $|z|=1$. Then the sets of hypercyclic
vectors for $T$ and $zT$ coincide.
\end{thmB}

For shortness, we denote by $\Lambda$ the set of all subsets of
$\Z_+$ of positive upper density. We also fix a base $\U$ of open
neighborhoods of 0 in $X$.

\medskip

\noindent{\bf Definition 2.} \ For $x\in X$ we denote
\begin{equation} \label{mx}
M^T_x=\bigl\{y\in X: \{n\in\Z_+:y-T^nx\in U\}\in\Lambda \ \
\text{for any $U\in\U$}\bigr\}.
\end{equation}
Clearly, $x\in\uu(T)$ if and only if $M^T_x=X$.

If additionally $X$ is complex, we denote
\begin{align} \label{nx}
&N^T_x=\bigl\{y\in X: \{n\in\Z_+:C_n(T,x)\cap (y+U)\neq
\varnothing\}\in \Lambda\ \ \text{for any $U\in\U$}\bigr\},
\\
\notag &\qquad\text{where}\ \ C_n(T,x)=\{zT^nx:z\in\T\}.
\end{align}

We start by establishing few straightforward properties of the sets
$M^T_x$ and $N^T_x$ defined in (\ref{mx}) and (\ref{nx}). The
following lemma is an elementary exercise. We leave the proof to the
reader.

\begin{lemma}\label{mx1} For any $x\in X$, $M^T_x$ is closed in $X$,
$T(M^T_x)\subseteq M^T_x$, $M^T_{Tx}=M^T_x$ and $M^T_{zx}=zM^T_x$
for each $z\in\K$. If $X$ is complex, then for any $x\in X$, $N^T_x$
is closed in $X$, $T(N^T_x)\subseteq N^T_x$, $N^T_{Tx}=N^T_x$ and
$N^T_{zx}=N^{zT}_x= N^T_x$ for any $z\in\T$.
\end{lemma}

\begin{lemma}\label{mx2} If $u,x,y\in X$ are such that
$y\in M^T_x$ and $x\in M^T_u$, then $y\in M^T_u$.
\end{lemma}

\begin{proof} Let $U\in \U$. Pick $V\in \U$ such that $V+V\subset U$.
Since $y\in M^T_x$, we can find $m\in\Z_+$ such that $y-T^mx\in V$.
Since $x\in M^T_u$, we have
$A=\{k\in\Z_+:x-T^ku\in(T^m)^{-1}(V)\}\in\Lambda$. Then for any
$k\in A$, we have
$$
y-T^{k+m}u=y-T^mx+T^m(x-T^ku)\in V+T^m((T^m)^{-1}(V))\subseteq
V+V\subseteq U.
$$
Thus $y-T^nu\in U$ for any $n\in m+A$. Since $m+A\in\Lambda$, we
have $y\in M^T_u$.
\end{proof}

\begin{lemma} \label{mx5} If $X$ is complex and $x\in N^T_x$, then $x\in
M^T_x$.
\end{lemma}

\begin{proof} Clearly the set $F=\{\mu\in\T:\mu x\in M^T_x\}$ is closed
since $M^T_x$ is. Let us verify that $F$ is non-empty. For each
$U\in\U$ consider
$$
F(U)=\bigl\{z\in\T:\{n\in\Z_+:zx-T^nx\in U\}\in\Lambda\bigr\}
$$
and $\overline{F}(U)$ be the closure of $F(U)$ in $\T$. Clearly
$F=\bigcap\{F(U):U\in\U\}$. Let us show that the sets $F(U)$ are
non-empty. Fix $U\in\U$. Choose $V\in \U$ and $\epsilon>0$ such that
$V+D_\epsilon x\subseteq U$, where
$D_\epsilon=\{z\in\C:|z|<\epsilon\}$. Since $x\in N^T_x$, there
exist $A\in\Lambda$ and a sequence $\{z_m\}_{m\in A}$ of elements of
$\T$ such that $z_mx-T^mx\in V$ for any $m\in A$. Let $k\in\N$ be
such that $k>\epsilon^{-1}$ and let $w_j=e^{2\pi ij/k}$ for $0\leq
j\leq k-1$. Then for any $m\in A$ we can pick $\nu(m)\in
\{0,\dots,k-1\}$ such that $|z_m-w_{\nu(m)}|<\epsilon$. Clearly $A$
is the union of $A_j=\{m\in A:\nu(m)=j\}$ for $0\leq j\leq k-1$.
Since the union of $A_j$ has positive upper density, there exists
$j\in\{0,\dots,k-1\}$ such that $A_j\in \Lambda$. For each $m\in
A_j$, we have
$$
w_jx-T^mx=z_mx-T^mx+(w_j-z_m)x\in V+D_\epsilon x\subset U.
$$
Hence $w_j\in F(U)$ and $F(U)$ is non-empty. From the definition of
the set $F(U)$, we also have that $\overline{F}(V)\subset J_\epsilon
F(V) \subseteq F(U)$, where $J_\epsilon=\{e^{it}:-\epsilon\leq t\leq
\epsilon\}$. Thus for each $U\in\U$ there is $V\in \U$ such that
$\overline{F}(V)\subset F(U)$. It follows that
\begin{equation}\label{fuv}
F=\bigcap_{U\in \U}F(U)=\bigcap_{U\in \U}\overline{F}(U).
\end{equation}
Finally, for any $U_1,\dots,U_n\in\U$, we have
$$
\varnothing\neq \beta(U_1\cap{\dots}\cap U_n)\subseteq
F(U_1)\cap{\dots}\cap F(U_n)\subseteq
\overline{F}(U_1)\cap{\dots}\cap\overline{F}(U_n).
$$
Thus $\{\overline{F}(U):U\in\U\}$ is a family of closed subsets of
the compact topological space $\T$, any finite subfamily of which
has nonempty intersection. Hence it has non-empty intersection and
(\ref{fuv}) implies that $F\neq\varnothing$.

Thus there is $z\in\T$ such that $zx\in M^T_x$. By Lemma~\ref{mx1},
then $z^2x\in M^T_{zx}$. According to Lemma~\ref{mx2}, the
inclusions $zx\in M^T_x$ and $z^2x\in M^T_{zx}$ imply $z^2x\in
M^T_x$. Proceeding inductively, we have $\{z^nx:n\in\N\}\subset
M^T_x$. That is, $Q_z=\{z^n:n\in\N\}\subseteq F$. Since $F$ is
closed $\overline{Q_z}\subseteq F$. Since $1\in \overline{Q_z}$ for
any $z\in\T$, we see that $1\in F$. That is, $x\in M^T_x$.
\end{proof}

\begin{lemma}\label{mx6} If $x\in X$ is such that $x\in M^T_x$, then
$M^T_x=F^T_x$, where $F^T_x=\overline{\{T^nx:n\in\Z_+\}}$.
\end{lemma}

\begin{proof} Clearly $M^T_x\subseteq F^T_x$. Since $x\in M^T_x$ and
$T(M^T_x)\subseteq M^T_x$, we see that $\{T^nx:n\in\Z_+\}\subseteq
M^T_x$. Since $M^T_x$ is closed, we have $F^T_x\subseteq M^T_x$.
Thus, $F^T_x=M^T_x$.
\end{proof}

\begin{lemma} \label{mx7} If $x\in X$ is such that $M^T_x$ has
non-empty interior, then there exists $y\in X$  such that $y\in
M^T_y=M^T_x$. \end{lemma}

\begin{proof} Since $M^T_x$ has non-empty interior, there is $k\in
\Z_+$ such that $T^k x$ is an interior point of $M^T_x$. In
particular, $y=T^kx\in M^T_x$. By Lemma~\ref{mx1},
$M^T_y=M^T_{T^kx}=M^T_x$. Thus $y\in M^T_x=M^T_y$.
\end{proof}

\subsection{Rotations}

\begin{proposition}\label{coc2} Let $X$ be a
complex topological vector space, $z\in\T$ and $T\in L(X)$. Then $T$
is $\uu$-frequently hypercyclic if and only if $zT$ is
$\uu$-frequently hypercyclic and $\uu(T)=\uu(zT)$.
\end{proposition}

\begin{proof} Let $x\in\uu(T)$. Then $M^T_x=N^T_x=X$, where the sets
$M^T_x$ and $N^T_x$ are defined in (\ref{mx}) and (\ref{nx}). On the
other hand, clearly $N^{zT}_x=N^T_x$. Hence $N^{zT}_x=X$. In
particular, $x\in N^{zT}_x$. By Lemma~\ref{mx5}, $x\in M^{zT}_x$.
Then by Lemma~\ref{mx6}, $M^{zT}_x=F^{zT}_x$, where $F^{zT}_x$  is
the closure of $\{(zT)^nx:n\in\Z_+\}$. Since $x$ is a hypercyclic
vector for $T$, Theorem~B implies that $x$ is a hypercyclic vector
for $zT$ and therefore $F^{zT}_x=X$. Hence $M^{zT}_x=X$ and
$x\in\uu(zT)$. Thus $\uu(T)\subseteq \uu(zT)$. Applying this
inclusion to $T$ replaced by $zT$ and $z$ replaced by $z^{-1}$, we
obtain $\uu(zT)\subseteq \uu(T)$. Thus $\uu(T)=\uu(zT)$.
\end{proof}

\subsection{Powers}

\begin{proposition}\label{coc3} Let $X$ be a
topological vector space and $T\in L(X)$. If $x\in X$ is such that
the set $M^T_x$, defined in $(\ref{mx})$, has non-empty interior,
then $x\in\uu(T)$.
\end{proposition}

\begin{proof} By Lemma~\ref{mx7}, we can pick $y\in X$ such that
$y\in M^T_y=M^T_x$. By Lemma~\ref{mx6}, $M^T_y=F^T_y$, where
$F^{T}_y$ is the closure of $\{T^ny:n\in\Z_+\}$. Thus, $F^T_y$ has
non-empty interior and Theorem~A implies that $F^T_y=X$. Since
$F^T_y=M^T_y=M^T_x$, we have $M^T_x=X$. That is, $x\in\uu(T)$.
\end{proof}

\begin{corollary} \label{pow} Let $X$ be a
topological vector space, $T\in L(X)$ be an $\uu$-frequently
hypercyclic operator and $n\in\N$. Then $T^n$ is also
$\uu$-frequently hypercyclic. Moreover $\uu(T)=\uu(T^n)$.
\end{corollary}

\begin{proof} Let $x\in \uu(T)$. Then the set $M^T_x$ defined in (\ref{mx})
coincides with $X$. Since the union of finitely many subsets of
$\Z_+$ has positive upper density if and only if at least one of
them has positive upper density, one can easily see that
$$
X=M^T_x=\bigcup_{j=0}^{n-1}M^{T^n}_{T^jx}.
$$
By Lemma~\ref{mx1}, the sets $M^{T^n}_{T^jx}$ are closed. Hence
there is $j\in \{0,\dots,n-1\}$ such that $M^{T^n}_{T^jx}$ has
non-empty interior. By Proposition~\ref{coc3}, $M^{T^n}_{T^jx}=X$.
That is, $T^jx\in\uu(T^n)$. Exactly as for usual hypercyclicity, one
can easily verify that if $y$ is an $\uu$-frequently hypercyclic
vector for $S\in L(X)$ and $R\in L(X)$ has dense range and commutes
with $S$, then $Ry\in\uu(S)$. Applying this observation with
$S=T^n$, $y=T^jx$ and $R=T^{n-j}$, we see that $T^nx\in\uu(T^n)$.
Since $x$ and $T^nx$ have the same orbits with respect to $T^n$ (up
to one element added or removed), we obtain $x\in\uu(T^n)$. Thus
$\uu(T)\subseteq \uu(T^n)$. The inclusion $\uu(T^n)\subseteq \uu(T)$
is obvious. \end{proof}

\begin{corollary} \label{min} Let $X$ be a
real topological vector space and $T\in L(X)$. Then $T$ is
$\uu$-frequently hypercyclic if and only if $-T$ is $\uu$-frequently
hypercyclic and $\uu(T)=\uu(-T)$.
\end{corollary}

\begin{proof} By Corollary~\ref{pow},
$\uu(T)=\uu(T^2)=\uu((-T)^2)=\uu(-T)$.
\end{proof}

\section{Proof of Theorems~\ref{hihi} and~\ref{spec}}

The above results allow us to strengthen Theorem~\ref{main}.

\begin{corollary} \label{tstar1} Let $X$ be a separable infinite
dimensional Banach space, $z\in\K$ and $T\in L(X)$ be such that
$|z|=1$ and there exists $f\in X^*\setminus\{0\}$ for which
$\limm_{n\to\infty}\|T^{*n}f\|^{1/n}=0$. Then $zI+T$ is not
$\uu$-frequently hypercyclic.
\end{corollary}

\begin{proof} By Corollary~\ref{tstar}, $I+z^{-1}T$ is not
$\uu$-frequently hypercyclic. Using Proposition~\ref{coc2} in the
case $\K=\C$ and Corollary~\ref{min} in the case $\K=\R$, we see
that $z(I+z^{-1}T)=zI+T$ is not $\uu$-frequently hypercyclic.
\end{proof}

Since $\|T^{*n}f\|\leq \|f\|\|T^{*n}\|=\|f\|\|T^n\|$ for any $f\in
X^*$ and any $n\in\Z_+$, we immediately obtain the following
slightly stronger form of Theorem~\ref{main}.

\begin{theorem}\label{main2} Let $T$ be a bounded linear operator on
a separable infinite dimensional Banach space $X$ such that
$\limm_{n\to \infty}\|T^n\|^{1/n}=0$ and $z\in\K$ be such that
$|z|=1$. Then $zI+T$ is not $\uu$-frequently hypercyclic. In
particular, $zI+T$ is not frequently hypercyclic.
\end{theorem}

In order to prove Theorems~\ref{hihi} and~\ref{spec}, we need the
following lemmas.

\begin{lemma} \label{HI0} Let $X\in\Phi_\K$ and $T\in L(X)$. Assume also
that $T$ has no non-trivial closed invariant subspaces of finite
codimension. Then $T=zI+S$, where $z\in \K$ and
$\limm_{n\to\infty}\|S^n\|^{1/n}=0$.
\end{lemma}

\begin{proof} Since $X\in\Phi_\K$, there are $z\in\K$ and $S\in {\cal
S}(X)$ such that $T=zI+S$. Let $X_\C=X$ and $R_\C=R$ for any $R\in
L(X)$ if $\K=\C$ and $X_\C$ be the complexification of $X$ and
$R_\C\in L(X_\C)$ be the complexification (=the unique complex
linear extension) of $R\in L(X)$ if $\K=\R$.

In any case $T_\C=zI+S_\C$ and $S_\C\in{\cal S}(X)$. The fact that
$T$ has no non-trivial closed linear subspaces of finite codimension
is clearly equivalent to the statement that the point spectrum
$\sigma_p(T_\C^*)$ of the dual $T_\C^*$ of $T_\C$ is empty.

Assume that $w\in \sigma(T_\C)\setminus\{z\}$. Then $w-z\in
\sigma(S_\C)\setminus\{0\}$. Since a non-zero element of the
spectrum of any strictly singular operator on a complex Banach space
is a normal eigenvalue \cite{pich}, we see that $w-z$ is a normal
eigenvalue of $S_\C$ and therefore $w-z\in \sigma_p(S_\C^*)$. Hence
$w\in\sigma_p(T_\C^*)$, which contradicts the equality
$\sigma_p(T_\C^*)=\varnothing$. Thus $\sigma(T_\C)=\{z\}$. Hence
$\sigma(S_\C)=\{0\}$. By the spectral radius formula
$\limm_{n\to\infty}\|S_\C^n\|^{1/n}=0$. Since $S$ is the restriction
of $S_\C$ to an $\R$-linear invariant subspace, $\|S^n\|\leq
\|S_\C^n\|$ for any $n\in\N$ and therefore
$\limm_{n\to\infty}\|S^n\|^{1/n}=0$.
\end{proof}

\begin{lemma} \label{HI} Let $X\in\Phi_\K$ and $T\in L(X)$ be a
hypercyclic operator. Then $T=zI+S$, where $z\in \K$, $|z|=1$ and
$\limm_{n\to\infty}\|S^n\|^{1/n}=0$.
\end{lemma}

\begin{proof} As it is shown in \cite{joh}, a hypercyclic operator
on a real or complex topological vector space has no non-trivial
closed invariant subspaces of finite codimension. Hence, by
Lemma~\ref{HI0}, $T=zI+S$, where $z\in \K$ and $\limm_{n
\to\infty}\|S^n\|^{1/n}=0$. It remains to demonstrate that $|z|=1$.
Using the equalities $T=zI+S$ and $\limm_{n
\to\infty}\|S^n\|^{1/n}=0$, one can easily verify that
\begin{equation}\label{ee1}
\lim_{n\to\infty}\|T^nx\|^{1/n}=|z|\ \ \text{for any $x\in
X\setminus\{0\}$.}
\end{equation}

Now if $|z|>1$, then according to (\ref{ee1}), $\|T^nx\|\to \infty$
as $n\to \infty$ for each non-zero $x\in X$. Similarly, if $|z|<1$,
then $\|T^nx\|\to 0$ as $n\to \infty$ for any $x\in X$. Hence $T$
can not be hypercyclic if $|z|\neq 1$. Thus $|z|=1$.
\end{proof}

Theorem~\ref{hihi} follows immediately from Lemma~\ref{HI} and
Theorem~\ref{main2}.

\begin{proof}[Proof of Theorem~\ref{spec}]
Let $T$ be an $\uu$-frequently hypercyclic operator on a complex
Banach space $X$. Assume also that $z$ is an isolated point of the
spectrum $\sigma(T)$. Let $x\in\uu(T)$, $P$ be the spectral
projection corresponding to the component $\{z\}$ of the spectrum of
$T$, $Y=P(X)$ and $S\in L(Y)$ be the restriction of $T$ to the
closed invariant subspace $Y$. Clearly $\sigma(S)=\{z\}$. Since
$S^nPx=PT^nx$ for each $n\in \Z_+$, we see that $Px\in\uu(S)$. Thus,
$S$ is an $\uu$-frequently hypercyclic operator on $Y$ satisfying
$\sigma(S)=\{z\}$. If $|z|<1$, then $\|S^ny\|\to 0$, while
$\|S^ny\|\to \infty$ if $|z|>1$ for any $y\in Y$ and $S$ can not be
hypercyclic. Hence $z\in\T$. Since $\sigma(S)=\{z\}$, we have
$S=zI+R$, where $\limm_{n\to\infty}\|R^n\|^{1/n}=0$. We have
obtained a contradiction with Theorem~\ref{main2}. \end{proof}

\section{Concluding remarks}

{\bf 1.} \ In the proofs of Proposition~\ref{coc2} and
Corollary~\ref{pow}, we use the following property of the family
$\Lambda$ of subsets of $\Z_+$ of positive upper density. Namely, if
a finite union of sets belongs to $\Lambda$, then one of them does.
The family $\Lambda_0$ of subsets of $\Z_+$ of positive lower
density fails to have this property. Thus our proof of stability of
the class of $\uu$-frequently hypercyclic operators under powers and
under multiplication by unimodular scalars does not work for the
class of frequently hypercyclic operators. This leads to the
following question.

\smallskip

{\bf Question 2.} \ Let $T$ be a frequently hypercyclic operator on
a separable infinite dimensional Banach space $X$. Is it true that
the operators $T^n$ and $zT$ for $n\in\N$ and $z\in\K$  with $|z|=1$
are frequently hypercyclic?

\smallskip

\noindent{\bf 2.} \ Proposition~\ref{coc2}, Corollary~\ref{pow} and
Corollary~\ref{min} admit the following generalization. Let $\Omega$
be any family of non-empty subsets of $\Z_+$ satisfying the
following conditions
\begin{align*}
&\text{if $A\in\Omega$ and $A\subseteq B\subseteq \Z_+$, then
$B\in\Omega$};
\\
&\text{if $A\in\Omega$ and $k\in\Z_+$, then $A+k\in\Omega$ and
$(A-k)\cap \Z_+\in \Omega$};
\\
\label{a3} &\!\!\!\begin{array}{ll}\text{if $k\in\N$ and $A_j$ for
$1\leq j\leq k$ are subsets of $\Z_+$ such that
\smash{$\bigcup\limits_{j=1}^kA_j\in\Omega$}}, \\ \text{then there
exists $j\in\{1,\dots,k\}$ such that $A_j\in \Omega$}.\end{array}
\end{align*}
We say that $x$ is an $\Omega$-{\it hypercyclic} vector for a
continuous linear operator $T$ on a topological vector space $X$ if
$\{n\in\Z_+:T^nx\in U\}\in\Omega$ for each non-empty open subset $U$
of $X$ and that $T$ is $\Omega$-{\it hypercyclic} if it has an
$\Omega$-hypercyclic vector.

All proofs in Section~\ref{secrot} work without any changes (just
replace $\Lambda$ by $\Omega$) for $\Omega$-hypercyclicity instead
of $\uu$-frequent hypercyclicity. Thus, the class of
$\Omega$-hypercyclic operators is stable under taking the powers and
under multiplying by unimodular scalars. Note also that when
$\Omega$ is the set of all infinite subsets of $\Z_+$, then
$\Omega$-hypercyclicity coincides with the usual hypercyclicity. The
following example seems to be interesting. Let $\Psi$ be the set of
all decreasing sequences $\phi=\{\phi_n\}_{n\in\Z_+}$ of positive
numbers such that $\summ_{n=0}^\infty \phi_n=\infty$. Clearly, for
each $\phi\in\Psi$, the family $\Omega_\phi$ of subsets $A$ of
$\Z_+$ for which $\summ_{n\in A}\phi_n=\infty$ satisfies all
properties from the above display. It is also easy to verify that an
operator $T$ is $\uu$-frequently hypercyclic if and only if it is
$\Omega_\phi$-hypercyclic for any $\phi\in\Psi$. This naturally
leads to the following question.

\smallskip

{\bf Question 3.} \ For which $\phi\in\Psi$, there is an
$\Omega_\phi$-hypercyclic operator on any separable infinite
dimensional Banach space?

\smallskip

\noindent{\bf 3.} \ One can observe that all operators $T$ {\it
proven} to be frequently hypercyclic have a stronger property.
Namely they admit a vector $x$ such that for any non-empty open set
$U$ the set $\{n\in\Z_+:T^nx\in U\}$ contains a subset, which has
positive density. Since there are sets $A\subset\Z_+$ of positive
lower density which contain no subsets of positive density, the last
property of $x$ is, formally speaking, stronger than to be a
frequently hypercyclic vector. It might be interesting to find out
whether all frequently hypercyclic operators admit vectors having
this stronger property.

\smallskip

\noindent{\bf 4.} \ Herrero \cite{herre} have characterized the
operator norm closure $\overline{{\tt HC}({\mathcal H})}$ of the set
${\tt HC}({\mathcal H})$ of hypercyclic operators on a separable
infinite dimensional complex Hilbert space $\mathcal H$. Later
Herrero and Wang \cite{herw} have shown that any $T\in
\overline{{\tt HC}({\mathcal H})}$ is a limit of a sequence $T_n$ of
elements of ${\tt HC}({\mathcal H})$ such that the differences
$T-T_n$ are compact. Using these results, one can easily prove the
following proposition.

\begin{proposition} \label{spec1} A nonempty compact subset $K$ of $\C$
is the spectrum of some bounded hypercyclic operator $T$ on a
separable infinite dimensional complex Hilbert space ${\mathcal H}$
if and only if $K\cup \T$ is connected.
\end{proposition}

\begin{proof} The 'only if' part is proved in \cite{herre}. Let $K$
be a non-empty compact subset of $\C$ such that $K\cup\T$ is
connected. We have to show that there is a hypercyclic operator
$T\in L({\mathcal H})$ such that $\sigma(T)=K$. Let $A$ be the set
of isolated points of $K$. Exactly as in the proof of
Theorem~\ref{spec}, one can easily see that $A\subset \T$. Clearly
$A$ is at most countable. Let $S\in L(\ell_2)$ be a compact weighted
backward shift: $Se_0=0$ and $Se_n=w_ne_{n-1}$ for $n\in\N$, where
$w=\{w_n\}_{n\in\N}$ is a sequence of non-zero complex numbers
tending to $0$ as $n\to\infty$. Consider the $\ell_2$-direct sum
$T_0$ of $zI+S$ for $z\in A$:
$$
T_0=\bigoplus_{z\in A}(zI+S).
$$
Then $\sigma(T_0)=\overline{A}$. From \cite{grish}[Theorem~5.1]  if
follows that $T_0$ is topologically mixing. Next, we take a bounded
normal operator $S_1$ on a separable complex Hilbert space
${\mathcal H}_1$ such that $\sigma(S_1)=K\setminus A$ and $S_1$ has
empty point spectrum. According to the mentioned results of Herrero
\cite{herre} and Herrero and Wang \cite{herw}, we have that $S_1\in
\overline{{\tt HC}({\mathcal H}_1)}$ and therefore there exists a
compact operator $K\in L({\mathcal H}_1)$ such that $T_1=S_1+K$ is
hypercyclic. For each $z\in \C\setminus \sigma(S_1)=\C\setminus
(K\setminus A)$, the operator $S_1-zI$ is invertible and therefore
by the Fredholm theorem, either $T_1-zI$ is invertible or $T_1^*-zI$
has non-trivial kernel (here $T_1^*$ is the dual operator, not the
Hilbert space adjoint). The latter case is not possible since the
dual of any hypercyclic operator has empty point spectrum. Thus
$\sigma(T_1)\subseteq \sigma(S_1)$. One the other hand, $S_1$ has
purely continuous spectrum and $T_1$ is a compact perturbation of
$S_1$. Hence $\sigma(S_1)\subseteq \sigma(T_1)$. Thus
$\sigma(T_1)=K\setminus A$. Let $T=T_0\oplus T_1$. Then
$\sigma(T)=\sigma(T_0)\cup\sigma(T_1)=\overline{A}\cup (K\setminus
A)=K$. Moreover $T$ is hypercyclic as a direct sum of a hypercyclic
operator with a mixing operator. \end{proof}

\smallskip

{\bf Question 4.} \ What compact subsets of $\C$ are the spectra of
frequently hypercyclic operators on $\ell_2$? What about
$\uu$-frequently hypercyclic operators?

\smallskip

{\bf Acknowledgements.} \ The author would like to thank the referee
for helpful comments and corrections.

\small\rm

\vskip1truecm

\scshape

\noindent Stanislav Shkarin

\noindent Queens's University Belfast

\noindent Department of Pure Mathematics

\noindent University road, Belfast, BT7 1NN, UK

\noindent E-mail address: \qquad {\tt s.shkarin@qub.ac.uk}

\end{document}